\documentclass[11pt]{amsart}
\usepackage{times}
\usepackage{amsfonts}
\usepackage{amsthm}
\usepackage{amsmath}
\usepackage{amssymb}
\advance \topmargin by -\headheight
\advance \topmargin by -\headsep
\evensidemargin \oddsidemargin
\newtheorem{theorem}{Theorem}[section]
\newtheorem*{mil}{Maximal Intersection Lemma}
\newtheorem{lemma}[theorem]{Lemma}
\newtheorem{corollary}[theorem]{Corollary}
\newtheorem{proposition}[theorem]{Proposition}
\newtheorem{remark}[theorem]{Remark}

\newtheorem{definition}[theorem]{Definition}

\theoremstyle{definition}
\newtheorem{example}{Example}[section]

\begin{document}
\title[The Bootstrap and von Neumann algebras : The Maximal Intersection Lemma]{The Bootstrap and von Neumann algebras : \\ The Maximal Intersection Lemma}
\author{Kenley Jung}
\begin{abstract} Given a suitably nested family $Z = \langle Z(m,k,\gamma) \rangle_{m,k \in \mathbb N, \gamma >0}$ of Borel subsets of matrices, and associated Borel measures and rate function, $\mu$, an entropy, $\chi^{\mu}(Z)$, is introduced which generalizes the microstates free entropy in free probability theory.  Under weak regularity conditions there exists a finite tuple of operators $X$ in a tracial von Neumann algebra such that
\begin{eqnarray*}
\chi^{\mu}(X) & \geq & \chi^{\mu}(X \cap Z) \\
                     & = & \chi^{\mu}(Z)\\ 
\end{eqnarray*}  
where $X \cap Z = \langle \Gamma(X;m,k,\gamma) \cap Z(m,k,\gamma) \rangle_{m, k \in \mathbb N, \gamma >0}$.  This observation can be used to establish the existence of finite tuples of operators with finite $\chi^{\mu}$-entropy.  The intuition and proof come from the bootstrap in statistical inference.   
\end{abstract}
\maketitle
\section{Introduction}

The main result of this paper is about von Neumann algebras and free probability.  Its proof and statement can be made entirely in terms of the language of those subjects.  However, the intuition behind these results comes from methods in statistical inference.  To provide motivation I will first survey the relevant statistical concepts and then relate them to operator algebras.

\subsection{Statistical Inference and the Bootstrap}

Suppose $x_1,\ldots, x_n$ is a sample of numbers.  They could be heights in a population, yes/no votes for a political proposition, or daily profits from a trading strategy.  From this limited sample, imagine you have to infer something about the entire "population" from which this sample was drawn (assume the sample size $n$ is significantly smaller than the total "population").  For instance, let's say we're interested in understanding the population mean.   We could average $\overline{x}$ of $x_1,\ldots, x_n$ and offer this as a best guess for the entire population mean.  But how accurately does $\overline{x}$ approximate the overall population mean?  If one got another random sample of numbers from the same source and computed the average couldn't it be different from the first average?  If so, how different?  Statistical inference provides a way of clarifying and quantifying this difference.

Suppose $X_1,\ldots, X_n$ are independent, identically distributed random variables and that $x_i = X_i(\omega)$ for $1 \leq i \leq n$.  Define $A=(X_1 + \cdots +X_n)/n$.   Notice that $\overline{x} = A(\omega)$.  Denote by $\mu$ and $\sigma$ the common mean and standard deviation of the $X_i$, respectively and by $N$ the standard normal distribution.  By the Central Limit Theorem
\begin{eqnarray*}
\frac{(A-\mu)}{\frac{\sigma}{\sqrt{n}}} \rightarrow N.
\end{eqnarray*}
If $P$ denotes the underlying probability measure, then $P(N \in (-2,2)) \sim 0.95$.  Thus, from the above,
\begin{eqnarray*}
0.95 & \sim & P\left (-2 < \frac{A - \mu}{\frac{\sigma}{\sqrt{n}}}  < 2 \right) \\ 
        & = & P\left (-\frac{2\sigma}{\sqrt{n}} < A - \mu < \frac{2\sigma}{\sqrt{n}} \right) \\
        & = & P\left(A -\frac{2\sigma}{\sqrt{n}} < \mu < A + \frac{2\sigma}{\sqrt{n}} \right).\\
\end{eqnarray*} 
If we're trying to capture the mean $\mu$ of the random variables, then the above says that, with $95\%$ probability, $\mu$ lies in the random interval $(A - 2\sigma/\sqrt{n}, A+ 2\sigma/\sqrt{n})$.  Evaluating $A$ at $\omega$ (and putting in an appropriate value for $\sigma$) yields a $95\%$-confidence interval.  It is the evaluation of a random process which will contain the (never seen) mean of the $X_i$ with $95\%$ probability. 

The above situation and the derivation of the confidence intervals is standard.  The quantity that one ends up trying to estimate, the mean $\mu$, is the most basic of parameters to estimate.  One could however imagine settings where one wants to estimate more complex parameters of the underlying random variables, parameters for which there is no clean analytic solution as above.  For example, the $x_i$ might be vectors so that the $X_i$ become random i.i.d. vectors and one might want to understand what fraction of the variance of the $x_i$ is "explained" by a proper subspace of fixed dimension (principal component analysis).  The math needed to arrive at an analytic solution for confidence intervals on the subspace variance or some other natural statistics become vastly harder, if not outright impossible.  But there is a method known as the \textbf{bootstrap} which replaces these difficult mathematical derivations with random simulation to arrive at an alternative solution.

In the example above consider the sample probability space $\Omega = \{x_1,\ldots, x_n\}$ equipped with the probability $P$ defined by $P((a,b)) = \#\{i: x_i \in (a,b)\}/n$.  Suppose the quantity of interest can be written as some function $s$ of a set of sample points.  In the first example above $s$ was the average of the sample points, but $s$ could be much more complicated.  A bootstrap sample $b_j$ consists of a random sample of size  $n$ drawn from $\Omega$ with replacement.  One can apply $s$ to this sample yielding $s(b_j)$; this is called a bootstrap replication.  Repeat this some fixed number of times, say $K=1000$, and consider the set of bootstrap replications $s(b_1),\ldots, s(b_{1000})$ associated to the bootstrap samples $b_1, \ldots, b_{1000}$.  The $95\%$ bootstrap confidence interval is the interval whose endpoints are the $5$ and $95$ percentile of these bootstrap replications.  I want to emphasize here that in place of mathematical formulas which bridge the data points ($\Omega$) to random variables, the bootstrap procedure uses only the data points, $\Omega$, and resamples to create its randomness.  It bypasses by simulation the (often difficult) math derivations as well as the parametric assumptions often needed to facilitate such computations.  This point is key in making an analogy between the bootstrap and a construction of von Neumann algebras with Borel subsets and measures.

The bootstrap was invented by Efron in the late 70's (see the introductory text \cite{e}), building on Quenouille's jackknife (\cite{q}).  Advances in computers and statistical software have made the bootstrap a widespread, practical tool in statistics and machine learning (e.g., bagging and Breiman's Random Forest, see \cite{ml} for a more detailed discussion).

\subsection{Free Probability and Microstates}


Voiculescu developed \textbf{free probability} as a kind of noncommutative probability theory modeled after a class of tracial von Neumann algebras called the free group factors.  In this probability theory random variables are replaced by operators in a tracial von Neumann algebra $M$, the expectation is replaced by $\varphi$, and independence is replaced by free independence.  Free probability has numerous parallels with classical probability and the fundamental result in \cite{v0} connects them by expressing free independence in terms of the asymptotics of independent random matrices.  

One parallel free probability has with classical probability is an entropy theory, a version of which is called the 'microstates free entropy'.  This incarnation is inspired by Boltzmann's entropy formula, $S = k \log W$, in statistical mechanics.  Here $S$ represents the entropy of a macrostate, $k$ is a constant, and $W$ is the "wahrscheinlichkeit" or probability obtained by counting the number of microstates which correspond to the macrostate.  In the free probability context the macrostate is replaced with a finite tuple $X$ of operators in the tracial von Neumann algebra $M$ and the microstates are replaced by matricial microstates which I'll define presently.  Given $m,k \in \mathbb C$ and $\gamma >0$, the $(m,k,\gamma)$ $*$-microstates $\Gamma(X;m,k,\gamma)$ consists of all elements $\xi = (\xi_1,\ldots, \xi_n) \in (M_k(\mathbb C))^n$ such that for any $1 \leq p \leq m$, $1 \leq i_1,\ldots, i_p \leq n$, and $j_1,\ldots, j_p \in \{1, *\}$,
\begin{eqnarray*}
|\varphi(x_{i_1}^{j_1} \cdots x_{i_p}^{j_p}) - tr_k(\xi_{i_1}^{j_1} \cdots \xi_{i_p}^{j_p})| < \gamma.
\end{eqnarray*}
$\xi$ is called a (matricial) microstate for the tuple $X$.  In the same way that one counts microstates to arrive at the entropy of the macrostate, one can deploy a layered limiting process on the matricial microstate spaces to obtain several entropy-like numerical measurements of $X$ (one of which is called the free entropy of $X$ and is denoted by $\chi(X)$).   Operator algebra applications of these entropy quantities hinge on the idea that the microstates as a totality reflect fundamental features of not only the tuple $X$, but also the von Neumann algebra it generates.   See \cite{v4} for further details.

As mentioned before, free probability regards operators as random variables.  Thus, the "macrostate" tuple $X$ is a tuple of 'random variables' and the matricial microstates can be regarded as sample points generated by this tuple.  The situation resembles the statistical inference setting discussed above.  There the idea was to extract from a sample of data points something about the mean of the random variables which generated them.  There were two ways to do this.  In the first, mathematical manipulations on the random variable level (approximation, Central Limit Theorem) allowed for a computation of a confidence interval.  The second used resampling of the data points, i.e. the bootstrap.  What does the bootstrap say in the free entropy setting when $X$ replaces the random variables and the matricial microstates are the sample points drawn from $X$?

Since the bootstrap mechanism operates entirely in the realm of the sample space with no parametric assumption on the random variables, the bootstrap analogue should be able to take subsets consisting of tuples of matrices and extract information from the tuple $X$ which "generated" them.  What conditions should be placed on the subsets and what "counting" function can be used (recall that one "counts" in the Boltzmann entropy, and ignoring details, one does the same with microstates entropy) to guarantee such an $X$ even exists?  

Only mild regularity conditions are needed to guarantee that such an $X$ exists.  If $Z = \langle Z(m,k,\gamma) \rangle_{m,k \in \mathbb N, \gamma >0}$ consists of suitably nested Borel subsets of $n$-tuples of $k\times k$ complex matrices which are uniformly bounded in operator norm over $m,k$ and $\gamma$, and $\mu$ consists of a family $\langle \mu(m,k,\gamma) \rangle_{m,k \in \mathbb N, \gamma >0}$ of Borel measures on the $Z(m,k,\gamma)$ and a rate function (which serves to measure the asymptotic decay in the "counting"), one can define a free entropy $\chi^{\mu}(Z) \in [-\infty, \infty]$ relative to $\mu$.  For a normalization of Lebesgue measure and a rate function of the form $k^{-2}$ this free entropy becomes the usual free entropy introduced in \cite{v1}.  It turns out that there exists a finite tuple $X$ in a tracial von Neumann algebra such that the intersection of its microstate spaces with $Z$ is equal to $\chi^{\mu}(Z)$.  In particular, the entropy of $X$ relative to $\mu$, $\chi^{\mu}(X)$, is greater than or equal to $\chi^{\mu}(Z)$.  This result is called the \textbf{maximal intersection lemma}.  It replaces the resampling methodology in the bootstrap with the ultraproduct construction of the hyperfinite $\mathrm{II}_1$-factor (Axiom of Choice).  The necessary terminology and proof of this lemma will be covered in the following section.
  
The maximal intersection lemma is an existence theorem which allows one to work with arbitrary Borel subsets of Euclidean space to create tuples of operators which have as a lower bound the entropy of the subsets.  The advantage here is that it can be easier to compute lower bound numerics (statistics) for families of subsets with some coarse condition (e.g. balls) than of those for the microstate spaces of some specific finite tuple of operators. 

The maximal intersection lemma can be used to produce tuples of operators which have finite free entropy quantities in several known situations.  For example, if $Z$ is the set of all selfadjoint matrices with operator norm no more than $1$, then the maximal intersection lemma produces a tuple $X$ with $\chi(X) \geq \chi(Z)$; invoking results of \cite{v3}, it turns out that $X$ is necessarily the free semicircular family on $n$ elements.  A similar result holds by taking Cartesian products of sets and invoking results of \cite{bd}.  Another application involves a concentration of measure type result on complements of freely independent families.  

\section{The Maximal Intersection Lemma}

\subsection{Scopes, Zones, and the lemma}

In order to state the maximal intersection lemma, it will be necessary to generalize some microstate notions and their associated free entropy quantities.  

For any $k \in \mathbb N$ $M_k(\mathbb C)$ denotes the $k \times k$ complex matrices and $tr_k$ is the normalized trace on $M_k(\mathbb C)$.   For an $n$-tuple $\xi = (\xi_1,\ldots, \xi_n)$ of elements in $M_k(\mathbb C)$, $\|\xi\|_{\infty}$ denotes the maximum of the operator norms of the $\xi_i$.  For any $n$, $R>0$ $(M_k(\mathbb C))^n$ denotes the space of $n$-tuples of elements in $M_k(\mathbb C)$ and $((M_k(\mathbb C))_R)^n$ denotes the subset of $(M_k(\mathbb C))^n$ of elements with $\|\xi\|_{\infty} \leq R$.  $(M_k(\mathbb C))^n$ is a finite dimensional real vector space and therefore has a unique topology generated by any of its norms.

\begin{definition} Fix $n \in \mathbb N$.  A zone (on $n$ variables) is a collection of sets $Z = \langle Z(m,k,\gamma) \rangle_{m,k \in \mathbb N, \gamma >0}$  such that the following conditions hold:
\begin{enumerate}
\item For any $m,k \in \mathbb N$ and $\gamma >0$, $Z(m,k,\gamma) \subset (M_k(\mathbb C))^n$ is a Borel subset.
\item For any $m_1 \geq m$ and $0<\gamma_1 \leq \gamma$, $Z(m_1,k,\gamma_1) \subset Z(m,k,\gamma)$.
\item For any fixed $*$-monomial $w$ in $n$ variables $\sup_{\xi \in Z(m,k,\gamma)} |tr_k(w(\xi))| < \infty$.  
\end{enumerate}
A zone $Z$ is \textit{bounded} by $R >0$ if for any $m,k$, $\gamma$, $Z(m,k,\gamma) \subset ((M_k(\mathbb C))_R)^n$ and $Z$ is said to be \textit{nonempty} if for any $m \in \mathbb N$ and $\gamma >0$ there exist  infinitely many $k$ such that $Z(m,k,\gamma) \neq \emptyset$.
\end{definition}

\begin{remark} If $Z = \langle Z(m,k,\gamma) \rangle_{m,k \in \mathbb N, \gamma >0}$ and $W = \langle W(m,k,\gamma) \rangle_{m,k \in \mathbb N, \gamma >0}$ are (bounded) zones, then the intersection 
\begin{eqnarray*}
W \cap Z = \langle W(m,k,\gamma) \cap Z(m,k,\gamma) \rangle_{m,k \in \mathbb N, \gamma >0},
\end{eqnarray*}  
union 
\begin{eqnarray*}
W \cup Z = \langle W(m,k,\gamma) \cup Z(m,k,\gamma) \rangle_{m,k \in \mathbb N, \gamma >0},
\end{eqnarray*}
and Cartesian product 
\begin{eqnarray*} (W,Z) = \langle W(m,k,\gamma) \times Z(m,k,\gamma) \rangle_{m,k \in \mathbb N, \gamma >0}
\end{eqnarray*}
are (bounded) zones.  The only property that requires some argument in this statement is condition (3) of Definition 2.1 for the union and Cartesian product of non-bounded zones.  This follows readily from induction on the length of $w$ and the Cauchy-Schwarz inequality.
\end{remark}

\begin{example} If $X$ is finite tuple of elements in a tracial von Neumann algebra, then $\Gamma(X)$ denotes the zone $\langle \Gamma(X;m,k,\gamma) \rangle_{m,k \in \mathbb N, \gamma >0}$ discussed in the introduction.  Similarly, if $R>0$, then $\Gamma_R(X)$ denotes the bounded zone $\langle \Gamma_R(X;m,k,\gamma) \rangle_{m,k \in \mathbb N, \gamma >0}$ where $\Gamma_R(X;m,k,\gamma) = \Gamma(X;m,k,\gamma) \cap ((M_k(\mathbb C))_R)^n$.  Given a zone $Z$, $X \cap Z = \Gamma(X) \cap Z$ and $X_R \cap Z = \Gamma_R(X) \cap Z$.
\end{example}

\begin{definition} A zone $W$ is a subzone of another zone $Z$ if for any $m \in \mathbb N$, $\gamma >0$ and sufficiently large $k$ dependent on $m$ and $\gamma$, $W(m,k,\gamma) \subset Z(m,k,\gamma)$.  This condition is denoted by $W \subset Z$.
\end{definition}

\begin{definition} Suppose for each $m, k \in \mathbb N$ and $\gamma >0$, $\mu_{m,k,\gamma}$ is a Borel measure on $(M_k(\mathbb C))^n$ which is finite on bounded subsets.  Assume $r=\langle r_k \rangle_{k=1}^{\infty}$ is a sequence of positive real numbers converging to $0$.  Write $\mu = (\langle \mu_{m,k,\gamma} \rangle_{m,k \in \mathbb N, \gamma >0},r)$.  $\mu$ is called a \textsl{scale} (on $n$ variables).  A pairing, $(Z, \mu)$, consisting of a zone $Z$ and scale $\mu$ on n variables is called a scope (on $n$ variables).
\end{definition}

\begin{definition}
Given a scope $(Z,\mu)$ on $n$ variables with $\mu = (\langle \mu_{m,k,\gamma} \rangle_{m,k \in \mathbb N, \gamma >0},r)$  as above, define successively

\[ \chi^{\mu} (Z(m,\gamma)) = \limsup_{k \rightarrow \infty} \left( r_k \cdot \log[\mu_{m,k,\gamma}(Z(m,k,\gamma))] \right),
\] 

\[ \chi^{\mu}(Z) = \inf \{\chi^{\mu}(Z(m,\gamma)): m \in \mathbb N, \gamma >0\}.  \]
$\chi^{\mu}(Z)$ is the free entropy of $Z$ relative to $\mu$.
\end{definition}

Suppose  $X$ is an $n$-tuple of selfadjoint elements in a tracial von Neumann algebra and  $Z = \Gamma(X)$ (Example 2.1).   Set $\mu_{m,k,\gamma} = c_k \cdot \text{vol}^{sa}$ where $\text{vol}^{sa}$ is Lebesgue measure w.r.t. the Euclidean norm $\|\cdot\|_2$ given by $\|\xi\|_2= \sum_{i=1}^n tr_k(\xi_i^*\xi_i)^{1/2}$, restricted to the subspace of $n$-tuples of selfadjoints and $c_k = k^{nk^2}$.   Set $r_k = k^{-2}$ and $r = \langle r_k \rangle_{k=1}^{\infty}$.  If $\mu = (\langle \mu_{m,k,\gamma} \rangle_{m,k \in \mathbb N, \gamma >0},r)$, then
\begin{eqnarray*}\chi^{\mu}(Z) = \chi(X)
\end{eqnarray*}
where $\chi$ denotes selfadjoint entropy introduced in \cite{v1}.  When $\mu$ consists of this normalization of Lebesgue measure and the standard rate function $r_k = k^{-2}$, then I'll write $\chi(Z)$ for $\chi^{\mu}(Z)$.  

\begin{remark} If $(W,\mu)$ is a subscope of $(Z,\mu)$, then $\chi^{\mu}(W) \leq \chi^{\mu}(Z)$.
\end{remark}
 
\begin{proposition} If $Z_1, \ldots, Z_n$ are zones and $\mu$ is a scaling, then 
\begin{eqnarray*}
\chi^{\mu}(\cup_{1 \leq j \leq n} Z_j) & = & \max_{1 \leq j \leq n} \chi^{\mu}(Z_j).\\
\end{eqnarray*}
\end{proposition}

\begin{proof}  $\mu = (\langle \mu_{m,k,\gamma} \rangle_{m,k \in \mathbb N, \gamma >0},r)$ for some Borel measures $\mu_{m,k,\gamma}$ and positive sequence $r_k$ converging to $0$.  For any $m, k \in \mathbb N$ and $\gamma >0$, subadditivity of measures yields,

\begin{eqnarray*} \mu_{m,k,\gamma}( \cup_{1 \leq j \leq n} Z_j(m,k,\gamma)) & \leq & \sum_{1 \leq j \leq n} \mu_{m,k,\gamma}(Z_j(m,k,\gamma)) \\                                                                  & \leq & n \cdot \max_{1 \leq j \leq n} \mu_{m,k,\gamma}(Z_j(m,k,\gamma)).\\
\end{eqnarray*}
Hence,

\begin{eqnarray*} \chi^{\mu}( \cup_{1 \leq j \leq n} Z_j;m,\gamma) & = &  \limsup_{k\rightarrow \infty} r_k \cdot \log \left [ \mu_{m,k,\gamma}  (\cup_{1 \leq j \leq k} Z_j(m,k,\gamma))) \right] \\ & \leq &  \limsup_{k\rightarrow \infty} r_k \log n + r_k \cdot \log \left( \max_{1 \leq j \leq n} \mu_{m,k,\gamma}(Z_j(m,k,\gamma)) \right) \\ & = & \limsup_{k \rightarrow \infty} r_k \cdot \log \left( \max_{1 \leq j \leq n} \mu_{m,k,\gamma}(Z_j(m,k,\gamma))\right) \\ & = & \max_{1 \leq j \leq n} \chi^{\mu}(Z_j;m,\gamma).\\
\end{eqnarray*}

\noindent Thus, for every $m$, $\chi^{\mu} (\cup_{1 \leq j \leq k} Z_j;m,m^{-1}) \leq \max_{1 \leq j \leq n} \chi^{\mu}(Z_j;m, m^{-1})$. For some fixed $1 \leq j_1 \leq n$ and infinitely many $m$,  $\chi^{\mu} (\cup_{1 \leq j \leq k} Z_j;m,m^{-1}) \leq  \chi^{\mu}(Z_{j_1};m, m^{-1})$.  Taking a limit and using Remark 2.6 yields

\begin{eqnarray*} \chi^{\mu}( \cup_{1 \leq j \leq k} Z_j) & \leq & \chi^{\mu}(Z_{j_1}) \\
                                                          & \leq & \max_{1 \leq j \leq n} \chi^{\mu}(Z_j) \\
                                                          & \leq &  \chi^{\mu}( \cup_{1 \leq j \leq k} Z_j).\\
\end{eqnarray*}
\end{proof}
 
Denote by $\mathbb W_n$ the collection of all $*$-monomials in $n$ indeterminates.  $\mathbb W_n$ is clearly countable.   For a fixed $w \in \mathbb W_n$ and $E \subset \mathbb C$ denote by $M_{w, E}(k,\gamma)$ the set of all $\xi=(\xi_1,\ldots, \xi_n) \in (M_k(\mathbb C))^n$ such that $tr_k(w(\xi))$ is in the $\gamma$-neighborhood of $E$.  If $Z = \langle Z(m,k,\gamma) \rangle_{m,k \in \mathbb N, \gamma >0}$ is a matricial scope then clearly 
\[
Z \cap M_{w, E} = \langle Z(m,k,\gamma) \cap M_{w, E}(k,\gamma) \rangle_{m,k \in \mathbb N, \gamma >0}\]

\noindent is a matricial scope.  When $E = \{\lambda\}$ for some $\lambda \in \mathbb C$ the quantities $M_{w,\{\lambda\}}(\ldots)$ will be written as $M_{w,\lambda}(\ldots)$.  In this case  $M_{w,\lambda}(k,\gamma)$ consists of all tuples whose $w$-moment is within $\gamma$ of $\lambda$.

\begin{lemma} Suppose $(Z, \mu)$ is a scope on $n$ variables.  If $w$ is a fixed $*$-monomial on $n$-variables, then for any $\epsilon >0$ there exists a complex number $\lambda$ bounded by a constant dependent only on $w$ such that 
$\chi^{\mu}(Z \cap M_{w,B(\lambda, \epsilon)}) = \chi^{\mu}(Z)$ 
\end{lemma}

\begin{proof} By definition there exists a constant $C$ dependent on $w$ such that for any $\xi \in Z(m,k,\gamma)$, $|tr_k(w(\xi))| < C$.  Find a cover of the closed ball of radius $C$ in $\mathbb C$ by open $\epsilon$-balls $B(\lambda_1, \epsilon), \ldots, B(\lambda_n,\epsilon)$.  For any $k$,

\[ Z(m,k,\gamma) \subset \cup_{i=1}^n M(w, B(\lambda_i, \epsilon), k, \gamma).
\] 

\noindent Hence,

\begin{eqnarray*} Z = \cup_{1 \leq j \leq n}  \left[ Z \cap M_{w,B(\lambda_j, \epsilon)} \right ].
\end{eqnarray*}

\noindent By Proposition 2.7 there exists some $1 \leq j\leq n$ such that $\chi^{\mu}(Z) = \chi^{\mu}(Z \cap M_{w, B(\lambda_j, \epsilon)})$. 
\end{proof}

\begin{lemma} If $(Z,\mu)$ is a matricial scope over $n$ variables, and $w \in \mathbb W_n$, then there exists a complex number $\lambda$ such that $\chi^{\mu}(Z \cap M_{w,\lambda}) =\chi^{\mu}(Z)$.
\end{lemma}

\begin{proof} By Lemma 2.8 there exists a constant $C$ dependent on $w$ such that for each $j \in \mathbb N$ there exists a $\lambda_j \in \mathbb C$, $|\lambda_j| \leq C$, satisfying $\chi^{\mu}(Z \cap M_{w,B(\lambda_j, j^{-1})}))) = \chi^{\mu}(Z)$.  Passing to a convergent subsequence assume without loss of generality that $\langle \lambda_j \rangle_{j=1}^{\infty}$ converges to some $\lambda_0 \in \mathbb C$.   Suppose $m \in \mathbb N$ and $\gamma >0$.  For $j$ sufficiently large, $j^{-1}, |\lambda_j - \lambda_0| < \gamma \Rightarrow M_{w, B(\lambda_j, j^{-1})}(k,\gamma) \subset M_{w,\lambda_0}(k, 3 \gamma)$.

\[  Z(m,k,\gamma) \cap M_{w, B(\lambda_j, j^{-1})}(k,\gamma) \subset Z(m,k,\gamma) \cap M_{w, \lambda_0}(k,3\gamma)
\]

\noindent Hence for $j$ sufficiently large, 

\begin{eqnarray*} \chi^{\mu}(Z) & = & \chi^{\mu}(Z \cap M_{w,B(\lambda_j, j^{-1})}) \\ & \leq & \chi^{\mu}((Z\cap M_{w, B(\lambda_j, j^{-1})})(m, \gamma)) \\ & \leq & \chi^{\mu}(Z \cap M_{w,\lambda_0})(m, 3\gamma)) \\
\end{eqnarray*}

\noindent Since $m$ and $\gamma$ were arbitrary, it follows that $\chi^{\mu}(Z) \leq \chi^{\mu}(Z \cap M_{w,\lambda_0})$.  The reverse inequality for $\chi^{\mu}$ is obvious. 
\end{proof}

\begin{lemma} If $(Z,\mu)$ is a scope over $n$-variables, then there exists a function $f: \mathbb W_n \rightarrow \mathbb C$ such that for any $k \in \mathbb N$,
\begin{eqnarray*}
\chi^{\mu}(Z \cap (\cap_{j=1}^k M_{w_j, f(w_j)})) & = & \chi^{\mu}(Z).
\end{eqnarray*}
\end{lemma}
\begin{proof}  Find an enumeration $\langle w_j \rangle_{j=1}^{\infty}$.  Apply Lemma 2.9 with $w = w_1$ to obtain a complex number $\lambda_1$ such that $\chi^{\mu}(Z \cap M_{w_1, \lambda_1}) = \chi^{\mu}(Z)$ and set $Z_1 = Z \cap M_{w_1,\lambda_1}$.  Apply Lemma 2.9 again with $w = w_2$ to the matricial scope $(Z_1,\mu)$ to obtain a complex number $\lambda_2$ such that $\chi^{\mu}(Z_1 \cap M_{w_2, \lambda_2}) = \chi^{\mu}(Z_1)$ and set $Z_2 = Z_1 \cap M_{w_2, \lambda_2}$.  Continue the process to arrive at a nested sequence of subzones $(Z_k, \mu)$ of $(Z,\mu)$ and a sequence $\langle \lambda_j \rangle_{j=1}^{\infty}$ such that for any $j$,
\begin{eqnarray*}
\chi^{\mu}(Z_j) & = & \chi^{\mu}(Z \cap (\cap_{i=1}^j M_{w_i, \lambda_i})) \\
		       & = & \chi^{\mu}(Z). \\
\end{eqnarray*}
Define $f: \mathbb W_n \rightarrow \mathbb C$ by $f(w_j) = \lambda_j$.  
\end{proof}

\begin{remark} Technically speaking in the proof above, one is using the Axiom of Dependent Choice to construct the $Z_k$ and $\lambda_k$.
\end{remark}

Fix a free ultrafilter $\mathcal F$ of $\mathbb N$ and denote by $\mathcal R$ a copy of the hyperfinite $\mathrm{II}_1$-factor.  Consider the associated ultraproduct of $\mathcal R$ by $\mathcal F$, $(\mathcal R^{\mathcal F}, \varphi^{\mathcal F})$.  Recall that $\mathcal R^{\mathcal F}$ is obtained by quotienting the von Neumann algebra of all uniformly bounded sequences $\ell^{\infty}(\mathcal R)$ by the ideal $J = \{ \langle x_i \rangle_{i=1}^{\infty} \in \ell^{\infty}(\mathcal R) : \lim_{i \in \mathcal F} \|x_i\|_2 =0\}$.  $\mathcal R^{\mathcal F}$ has a canonical trace $\varphi^{\mathcal F}$ given by $\varphi^{\mathcal F}(Q(x)) = \lim_{i \in \mathcal F} \varphi(x_i)$, $x = \langle x_i \rangle_{i=1}^{\infty}$ where $Q: \ell^{\mathcal R} \rightarrow \mathcal R^{\mathcal F}$ is the quotient map.  This definition is independent of the choice of representative $x$.

Because $\mathcal F$ is free, if $\langle c_j \rangle_{j=1}^{\infty}$ is a sequence in $\mathbb C$ which converges to $c$, then $\lim_{i \in \mathcal F} c_i = c$.
\begin{mil} If $(Z,\mu)$ is a bounded scope on $n$ variables, then there exists a tracial von Neumann algebra with a finite $n$-tuple of generators $X$, such that 
\begin{eqnarray*}
\chi^{\mu}(X) & \geq & \chi^{\mu}(X \cap Z) \\
                      & = & \chi^{\mu}(Z).
\end{eqnarray*} 
The norm of the operators in $X$ can be arranged to have a bound no greater than that of $Z$.
\end{mil}

\begin{proof}  The conclusion is vacuous if $\chi^{\mu}(Z) = -\infty$ so it suffices to prove this under the additional assumption that $\chi^{\mu}(Z) > -\infty$.  Invoke Lemma 2.10 to produce a function $f: \mathbb W_n \rightarrow \mathbb C$ such that for any $m \in \mathbb N$,
\begin{eqnarray*}
\chi^{\mu}(Z \cap (\cap_{j=1}^m M_{w_j, f(w_j)})) & = & \chi^{\mu}(Z) \\
                                                                             & > & -\infty.
\end{eqnarray*}
For any $m \in \mathbb N$, $1 \leq j \leq m$, and infinitely many $k$,

\[Z(m, k, m^{-1}) \cap \left ( \cap_{j=1}^m M_{w_j, f(w_j)}(k, m^{-1}) \right)  \neq \emptyset.
\]

\noindent It follows that for each $m$ there exists a $k_m \in \mathbb N$ and an $n$-tuple $\xi_m \in Z(m, k_m, m^{-1}) \cap ( \cap_{j=1}^m M_{w_j, f(w_j)}(k_m, m^{-1}))$ of $k_m \times k_m$ matrices such that for all $1 \leq j \leq m$, 
\begin{eqnarray*}
|tr_k(w_j(\xi_m))-f(w_j)| & < & m^{-1}.
\end{eqnarray*}  

For each $m$ fix a trace preserving $*$-embedding $\pi_m:M_{k_m}(\mathbb C) \rightarrow \mathcal R$.  For each $1 \leq i \leq n$, denote by $\xi_{i,m}$ the ith coordinate of $\xi_m$ so that $\xi_m = (\xi_{1,m}, \ldots, \xi_{n,m})$.  Define $y_i = \langle \pi_m(\xi_{i,m}) \rangle_{m=1}^{\infty}$.  Since $Z$ is a bounded zone, $y_i \in \oplus_{k=1}^{\infty} \mathcal R$ and the norms of the $y_i$ are bounded by the bound on $Z$.  Notice that by construction, for any $j \in \mathbb N$
\begin{eqnarray*}
\lim_{m \rightarrow \infty} \varphi\left[w_j(\pi_m(\xi_{1,m}), \ldots, \pi_m(\xi_{n,m}))\right] & = & \lim_{m \rightarrow \infty} \varphi(\pi_m(w_j(\xi_{1,m}, \ldots, \xi_{n,m}))) \\
& = & \lim_{m \rightarrow \infty} tr_{k_m}(w_j(\xi)) \\
                                                                                                                               & = & f(w_j). \\
\end{eqnarray*}

Recall the free ultrafilter $\mathcal F$ of $\mathbb N$ and the associated ultraproduct of $\mathcal R$ by $\mathcal F$, $(\mathcal R^{\mathcal F}, \varphi^{\mathcal F})$.  Denote by $x_i$ the image of $y_i$ in $(\mathcal R^{\mathcal F}, \varphi^{\mathcal F})$.  From freeness of the ultrafilter $\mathcal F$ and the above computation it follows that for any $j \in \mathbb N$,
\begin{eqnarray*} 
\varphi^{\mathcal F}(w_j(x_1, \ldots, x_n)) & = & \lim_{m \in \mathcal F} \varphi \left [w_j\left(\pi_m(\xi_{1, m}), \ldots, \pi_m(\xi_{n, m}) \right) \right]\\
                                                               & = & f(w_j).\\
\end{eqnarray*}

\noindent Set $X = (x_1, \ldots, x_n)$.  

$X$ is an $n$-tuple of elements in the $\mathrm{II}_1$-factor $(R^{\lambda}, \varphi^{\lambda})$ whose $w$-moment for any $w \in \mathbb W_n$ is $f(w)$.  Moreover, for any fixed $m \in \mathbb N$ there exists $p \in \mathbb N$ large enough so that for all $k$,

\begin{eqnarray*} \Gamma(X;m,k,m^{-1}) & \supset & Z(p,k,p^{-1}) \cap \Gamma(X;m,k,m^{-1}) \\ & \supset & Z(p, k, p^{-1}) \cap \left ( \cap_{j=1}^p M_{w_j, f(w_j)}(k, p^{-1}) \right).\\
\end{eqnarray*}

\noindent  Applying $r_k \cdot \log \mu(m,k,\gamma)$, followed by a $\limsup_{k \rightarrow \infty}$ yields

\begin{eqnarray*} \chi^{\mu}(X;m,m^{-1})  & \geq & \chi^{\mu}(Z \cap X;m,m^{-1}) \\
                                           & \geq & \chi^{\mu}(Z \cap (\cap_{j=1}^p M_{w_j, f(w_j)});p, p^{-1})\\
                                           & \geq & \chi^{\mu}(Z \cap (\cap_{j=1}^p M_{w_j, f(w_j)})) \\
                                           & = & \chi^{\mu}(Z).
\end{eqnarray*}
This being true for any $m$ it follows that 
\begin{eqnarray*}
\chi^{\mu}(X) & \geq & \chi^{\mu}(X \cap Z) \\
                     & \geq & \chi^{\mu}(Z).\\
\end{eqnarray*}  
$\chi^{\mu}(Z) \geq \chi^{\mu}(X \cap Z)$ so by the above $\chi^{\mu}(X \cap Z) = \chi^{\mu}(Z)$, completing the proof. 
\end{proof}

\subsection{Examples}

As mentioned in the introduction, the Maximal Intersection Lemma can be used to manufacture tuples of operators from Euclidean subsets in such a way that the von Neumann algebra generated from the tuples inherit some geometric-measure-theoretic properties of the Euclidean subsets.  I'll discuss here a few simple examples of this.  The first is an existential way of getting a nontrivial free entropy example without resorting to the multivariate random matrix techniques in \cite{v0}.

\begin{example} (\textbf{nontriviality of $\chi$ for $n$-tuples, $n >1$}) Fix $n \in \mathbb N$ and for each $m,k \in \mathbb N$, $\gamma >0$ define $Z^n(m,k,\gamma)$ to be the Cartesian product of the ball of operator norm radius $1$ in $M^{sa}_k(\mathbb C)$ $n$ times.  $Z^n = \langle Z^n(m,k,\gamma) \rangle_{m,k\in \mathbb N, \gamma >0}$ is a bounded zone.  Recall that a $(0,1)$ semicircular element is a contractive, selfadjoint element $s$ in a tracial von Neumann algebra whose $k$th moment under the trace is $\frac{2}{\pi}\int_{-1}^1 t^k \sqrt{1-t^2}\, dt$.  Voiculescu's formula for the free entropy of a single selfadjoint element in \cite{v1} shows that $\chi(s) = \frac{1}{2} \log (2\pi e)$ and clearly $\chi(Z^1) \geq \chi(s)$.  Alternatively, one can use the change of variables in \cite{m} or \cite{sr} to show that $\chi(Z^1) \geq  \frac{1}{2} \log (2\pi e)$.  In any case,
\begin{eqnarray*} \chi(Z^n) & = &   \limsup_{k \rightarrow \infty} \left[ k^{-2} \cdot  \log(\text{vol}(Z^n(m,k,\gamma))) + n \log k \right] \\
                          & \geq &   n \lim_{k \rightarrow \infty} \left[ k^{-2} \cdot \log(\text{vol}(\Gamma_1(s;m,k,\gamma)+ n \log k \right] \\
& = & \frac{n}{2} \cdot \log(2\pi  e) \\ & > & -\infty. \\
\end{eqnarray*} 
Invoke the Maximal Intersection Lemma to produce an $n$-tuple of selfadjoint elements $X$ in a tracial von Neumann algebra such that $\chi(X) \geq \chi(X\cap Z^n) = \chi(Z^n) = \frac{n}{2} \cdot \log(2\pi  e) > -\infty$.  

   Using the maximal entropy result in \cite{v5}, this $n$-tuple must in fact be a free $n$-semicircular family, and thus, in the context of Cartesian products, the limiting process in the Maximal Intersection Lemma which produces the tuple $X$ invariably leads to a free $n$-semicircular family.  
\end{example}

A stronger statement about how the semicircular family saturates the scope of selfadjoint contractions is possible. Applying the argument with slightly more care than in the situation above yields:

\begin{corollary} Denote by $S_n$ an $n$-tuple of freely independent, semicircular elements.  Fix $m_0 \in \mathbb N$ and $\gamma_0 >0$. There exists a $c <1$ such that for sufficiently large $k$, 

\begin{eqnarray*} \frac{\text{vol}\left [\Gamma_1(S_n;m_0,k,\gamma_0)^c \cap ((M^{sa}_k(\mathbb C))_1)^n\right]}{\text{vol}(\Gamma_1(S_n;m_0,k,\gamma_0))} \leq c^{k^2} 
\end{eqnarray*}
\end{corollary}

\begin{proof} For any $m,k\in \mathbb N$ and $\gamma >0$ define 
\begin{eqnarray*} Z(m,k,\gamma) = ((M^{sa}_k(\mathbb C))_1)^n \cap \Gamma_1(S_n;m_0,k,\gamma_0)^c.
\end{eqnarray*} 
  $Z = \langle Z(m,k,\gamma) \rangle_{m,k \in \mathbb N, \gamma >0}$ is a bounded zone on $n$ variables.  I claim that $\chi(Z) < \frac{n}{2} \cdot \log(2\pi  e)$.  Indeed, suppose to the contrary that this is not the case.  Then $\chi(Z) \geq \frac{n}{2} \cdot \log(2\pi  e)$ and by the volume formula of the unit ball in Euclidean space and Stirling's formula, it follows that $\chi(Z) = \frac{n}{2} \cdot \log(2\pi  e)$.  Invoke the Maximal Intersection Lemma to produce an $n$-tuple of contractions $X$ such that 
\begin{eqnarray*}
\chi(X) & \geq & \chi(X \cap Z) \\
            & \geq & \chi(Z) \\
            & = & \frac{n}{2} \cdot \log(2\pi  e).\\
\end{eqnarray*}  
Again by the volume formula for the unit ball in Euclidean space and Stirling's formula (or now using the general bound for free entropy \cite{v1}) and the fact that the elements of $X$ are contractions, $\chi(X) \leq  \frac{n}{2} \cdot \log(2\pi  e)$ so that $\chi(X) = \frac{n}{2} \cdot \log(2\pi  e)$.  By the maximal entropy result in \cite{v5}, $X$ is a free family of $n$-semicircular elements.   $\chi(X \cap Z) > -\infty$ so in particular, for sufficiently large $k$, 
\begin{eqnarray*}
\emptyset & \neq & \Gamma_1(X;m_0,k,\gamma_0) \cap Z(m_0,k,\gamma_0) \\
                 & = & \Gamma_1(S_n;m_0,k,\gamma_0) \cap Z(m_0,k,\gamma_0).  \\
\end{eqnarray*}
This is preposterous in light of the definition of $Z$.   $\chi(Z) < \frac{n}{2} \cdot \log(2\pi  e)$.

Set $a = \chi(Z) -  \frac{n}{2} \cdot \log(2\pi  e)$.  By the above, $a <0$.  Using the regularity of $S_n$ (in the sense that replacing a $\limsup_{k \rightarrow \infty}$ with a $\liminf_{k \rightarrow \infty}$ in the definition of $\chi(S_n)$ yields the same quantities, see \cite{v3}),  

\begin{eqnarray*} 0 & > & a \\
                    & = & \chi(Z) - \chi(S_n) \\
                    & \geq & \left (\limsup_{k \rightarrow \infty} k^{-2} \cdot \log(\text{vol}(Z(m_0,k,\gamma_0)) + n \log k \right)- \\
                    &   &  \left (\liminf_{k \rightarrow \infty} k^{-2} \cdot \log(\text{vol}(\Gamma_1(S_n;m_0,k,\gamma_0) + n \log k\right) \\
                    & \geq & \limsup_{k \rightarrow \infty} k^{-2} \cdot \log \left [ \frac{\text{vol}(Z(m_0,k,\gamma_0))}{\text{vol}(\Gamma_1(S_n:m_0,k,\gamma_0))} \right ] \\
                    & \geq & \limsup_{k \rightarrow \infty} k^{-2} \cdot \log \left[ \frac{\text{vol}\left [\Gamma_1(S_n;m_0,k,\gamma_0)^c \cap ((M^{sa}_k(\mathbb C))_1)^n\right]}{\text{vol}(\Gamma_1(S_n:m_0,k,\gamma_0))} \right ]
\end{eqnarray*}

Hence, for sufficiently large $k$,

\begin{eqnarray*} \frac{\text{vol}\left [\Gamma_1(S_n;m_0,k,\gamma_0)^c \cap ((M^{sa}_k(\mathbb C))_1)^n\right]}{\text{vol}(\Gamma_1(S_n:m_0,k,\gamma_0))} & \leq & e^{ak^2}.
\end{eqnarray*}

\noindent Set $c = e^a < 1$.
\end{proof}

\begin{remark} The exponential decay in the corollary should remind one of a normal Levy family in the context of concentration. Alternatively, one can proof the corollary above by using such techniques as in \cite{v3} (essentially by accounting for the Gaussian dimensional decay).
\end{remark}

\begin{example} (\textbf{Nontriviality of $\chi$ in Cartesian Products}) This is proceeds much like the semicircular case.  Fix $n \in \mathbb N$ and suppose $X_1, \ldots, X_n$ are tuples of selfadjoint elements in a tracial von Neumann algebra with finite free entropy.  Assume moreover that each $X_i$ is regular w.r.t. free entropy, i.e., using a $\limsup_{k\rightarrow \infty}$ or $\liminf_{k\rightarrow \infty}$ in the definition of $\chi$ results in the same quantity.  Fix an $R$ greater than the operator norms of any of the elements in the $X_i$.  Define $Z(m,k,\gamma) = \Pi_{i=1}^n \Gamma(X_i;m,k,\gamma)$ and $Z = \langle Z(m,k,\gamma) \rangle_{m,k \in \mathbb N, \gamma >0}$.  $Z$ is a bounded matrical scope and may be regarded as a zone with the usual free entropy.  By the regularity of the $X_i$,

\begin{eqnarray*} \chi(Z) & = & \chi(X_1) + \cdots + \chi(X_n).
\end{eqnarray*} 

\noindent The Maximal Intersection Lemma provides finite tuples of selfadjoints, $Z_1,\ldots, Z_n$ in a common tracial von Neumann algebra $M$ with $\#Z_i = \#X_i$ such that $\chi(Z_1,\ldots, Z_n) \geq \chi((Z_1, \ldots, Z_n) \cap Z) = \chi(Z) = \chi(X_1) + \cdots + \chi(X_n)$.  By the definition of $Z$ the noncommutative moments of $Z_i$ are the same as those of $X_i$ and thus, there exists a $*$-isomorphism $\pi_i$ from the von Neumann algebra generated by $X_i$ onto the von Neumann subalgebra of $M$ generated by $Z_i$.  Thus, 

\begin{eqnarray*}  \chi(X_1) + \cdots + \chi(X_n) & \leq & \chi(Z_1,\ldots, Z_n) \\
                                                  & = & \chi(\pi_1(X_1)),\ldots, \pi_n(X_n)) \\ 
                                                  & \leq & \chi(\pi_1(X_1)) + \cdots + \chi_n(X_n) \\
                                                  & = & \chi(X_1) + \cdots + \chi(X_n).
\end{eqnarray*}
By \cite{bd} the tuples $Z_i$ are in fact free and again, Cartesian products in the limiting process of the Maximal Intersection Lemma necessarily leads to freeness.  
\end{example}

There is here as well a concentration type result on the microstates.  I'll omit the proof and simply state the corollary, which as before has an alternate proof by accounting for the Gaussian decay:

\begin{corollary} Suppose $1 \leq i \leq n$ and $X_i$ is a regular, finite tuple of selfadjoint elements such that $\chi(X_i) > -\infty$.  Set $R_i = \max_{x \in X_i} \|x\|$ and fix $m_0 \in \mathbb N$ and $\gamma_0 >0$.  Define

\begin{eqnarray*} F(m_0,k, \gamma_0) =\{(\xi_1,\ldots, \xi_n) \in \Pi_{i=1}^n \Gamma_{R_i}(X_i;m_0,k,\gamma_0): \{\xi_1,\ldots, \xi_n\} \text{ are } (m_0,\gamma_0)\text{-free}\}.  
\end{eqnarray*}

\noindent There exists a $c<1$ such that for sufficiently large $k$,

\begin{eqnarray*} \frac{\text{vol}\left[F(m_0,k,\gamma_0)^c \cap \Pi_{i=1}^n \Gamma_{R_i}(X_i;m_0,k,\gamma_0)\right]}{\text{vol}\left[\Pi_{i=1}^n \Gamma_{R_i}(X_i;m_0,k,\gamma_0)\right]} < c^{k^2}.
\end{eqnarray*}
\end{corollary}


\begin{thebibliography}{9}

\bibitem{bd} Philippe Biane, Yoann Dabrowksi "Concavification of free entropy", http://arxiv.org/pdf/1201.0716.pdf.

\bibitem{e} Bradley Efron, Robert J. Tibshirani 'An Introduction to the Bootstrap", Chapman \& Hall, 1994.

\bibitem{ml} Trevor Hastie, Robert Tibshirani, Jerome Friedman 'The Elements of Statistical Learning', Springer, (2009), 2nd. Edition

\bibitem{m} M.L. Mehta, "Random Matrices, 3rd ed", New York: Academic Press, 2004.

\bibitem{q} Quenouille, M. 'Approximate tests of correlation in time series'.  J. Royal. Statist. Soc. B 11, 18-44.

\bibitem{sr} Jean Saint Raymond 'Le volume des ideaux d'operateurs classiques', Studia Mathematica, T. LXXX (1984), 63-75.

\bibitem{v0} Dan-Virgil Voiculescu 'Limit laws for random matrices and free products,' Inventiones Mathematicae 104, (1991), 201-220.

\bibitem{v1} Dan-Virgil Voiculescu 'The analogues of entropy and of Fisher's 
information measure in free probability theory, II', Inventiones 
Mathematicae 118, (1994), 411-440.

\bibitem{v5} Dan-Virgil Voiculescu 'The analogues of entropy and of Fisher's 
information measure in free probability theory IV: Maximum Entropy and Freeness', Free Probability Theory, American Math Society (1997), 293-302.

\bibitem{v3} Dan-Virgil Voiculescu 'A strengthened asymptotic freeness result for random matrices with applications to free entropy'. IMRN 1, (1998) 41-64.

\bibitem{v4} Dan-Virgil Voiculescu 'Free Entropy', Bull. Lond. Math. Soc. 34 (3), (2002) 257-332.

\end{thebibliography}
\end{document}